\documentclass[12pt]{article}

\usepackage{amsfonts}
\usepackage{amssymb,amsmath,amsthm}
\usepackage{latexsym}
\usepackage{fullpage}
\usepackage{hyperref}

\newtheorem{theorem}{Theorem}[section]
\newtheorem{prop}[theorem]{Proposition}

\newtheorem{lemma}[theorem]{Lemma}

\newtheorem{defn}[theorem]{Definition}

\newtheorem{claim}{Claim}

\theoremstyle{definition}
\newtheorem{remark}{Remark}

\newcounter{tenumerate}

\newenvironment{tenumerate}{\begin{list}{\bf\alph{tenumerate})}{\usecounter{tenumerate}}}{\end{list}}

\newcommand{\Ber}[1]{{\mathcal A_{#1}}}

\renewcommand{\epsilon}{\varepsilon}
\newcommand{\eps}{\varepsilon}

\renewcommand{\dim}{\mathsf{dim}}
\newcommand{\codim}{\mathsf{codim}}

\newcommand{\of}[1]{\left( #1 \right)}

\newcommand{\R}{{\mathbb R}}
\newcommand{\N}{{\mathbb N}}
\newcommand{\F}{{\mathbb F}}
\newcommand{\remove}[1]{}
\renewcommand{\le}{\leqslant}
\renewcommand{\ge}{\geqslant}
\renewcommand{\leq}{\leqslant}
\renewcommand{\geq}{\geqslant}

\newcommand{\sasha}[1]{{\tt \bf \% #1\%}}

 \title{{\bf Almost Euclidean subspaces of
    $\ell_1^N$ via expander codes}\thanks{A preliminary version of this paper appeared in the {\em Proceedings of the 19th Annual ACM-SIAM Symposium
               on Discrete Algorithms}, January 2008.}}

\author{{Venkatesan Guruswami}\thanks{University of Washington, Department
      of Computer Science and Engineering, Box 352350, Seattle, WA 98195. Part of
    this work was done when the author was on leave at the School of
    Mathematics, Institute for Advanced Study, Princeton, NJ.
    Research supported in part by NSF CCF-0343672, a Packard Fellowship, and NSF grant CCR-0324906 to the IAS. {\tt venkat@cs.washington.edu}}
\and
{James R. Lee}\thanks{University
    of Washington, Department of Computer Science and Engineering,
    Seattle, WA 98195.  Research supported in part by NSF CCF-0644037.
{\tt jrl@cs.washington.edu}}
\and
{Alexander Razborov}\thanks{Institute for Advanced Study, School of Mathematics,
    Princeton, NJ and Steklov Mathematical Institute,
    Moscow, Russia. Current address: University of Chicago, Department of Computer Science, Chicago, IL 60637. {\tt razborov@cs.uchicago.edu}}}

\begin{document}
\date{}
\maketitle
\thispagestyle{empty}

\begin{abstract}
  We give an explicit (in particular, deterministic polynomial time)
  construction of subspaces $X \subseteq \R^N$ of dimension
  $(1-o(1))N$ such that for every $x \in X$,
  $$(\log N)^{-O(\log\log\log N)} \sqrt{N}\, \|x\|_2 \leq \|x\|_1 \leq \sqrt{N}\, \|x\|_2.$$
If we are allowed to use $N^{1/\log\log N}\leq N^{o(1)}$ random bits
and $\dim(X) \geq (1-\eta)N$ for any fixed constant $\eta$, the lower bound
can be further improved to $(\log N)^{-O(1)} \sqrt{N}\|x\|_2$.

\smallskip Through known connections between such Euclidean sections
of $\ell_1$ and compressed sensing matrices, our result also gives
explicit compressed sensing matrices for low compression factors for
which basis pursuit is guaranteed to recover sparse signals.  Our
construction makes use of unbalanced bipartite graphs to impose local
linear constraints on vectors in the subspace, and our analysis relies
on expansion properties of the graph. This is inspired by similar
constructions of error-correcting codes.
\end{abstract}

\smallskip
\noindent {\it Mathematics Subject Classification (2000) codes:} 68R05, 68P30,
51N20.

\newpage
\section{Introduction}

Classical results in high-dimensional geometry \cite{FLM77,kasin}
state that a random (with respect to the Haar measure)
subspace $X \subseteq \mathbb R^N$ of dimension $\epsilon N$ \cite{FLM77}
or even $(1-\epsilon)N$ \cite{kasin}
is an almost Euclidean section in $\ell_1^N$, in the sense
that $\sqrt N \|x\|_1$ and $\|x\|_2$ are within constant
factors, uniformly for every $x \in X$.
Indeed, this is a particular example of the use of
the probabilistic method, a technique which is now
ubiquitous in asymptotic geometric analysis.

On the other hand, it is usually the case that objects constructed
in such a manner are very hard to come by {\em explicitly.}
Motivated in part by ever growing connections with combinatorics
and theoretical computer science, the problem of explicit
constructions of such subspaces
has gained substantially in popularity over the last several years;
see, e.g. \cite[Sec. 4]{SzarekICM}, \cite[Prob. 8]{MilmanGAFA},
\cite[Sec. 2.2]{JShandbook}.
Indeed, such subspaces (viewed as embeddings)
are important for
problems like high-dimensional nearest-neighbor search \cite{Indyk1}
and compressed sensing \cite{donoho}, and one expects
that explicit constructions will lead, in particular,
to a better understanding of the underlying geometric structure.
(See also the end of the introduction for a discussion
of the relevance to compressed sensing.)

\subsection{Previous results and our contributions}
If one relaxes the requirement that $\dim(X) = \Omega(N)$ or allows a
limited amount of randomness in the construction, a number of results
are known.  In order to review these, we define the {\em distortion
  $\Delta(X)$ of $X \subseteq \mathbb R^N$} by
$$\Delta(X)= \sqrt{N} \cdot \max_{0 \neq x \in X} \frac{\|x\|_2}{\|x\|_1}.$$
In the first direction,
it is well-known that an explicit construction with distortion $O(1)$
and $\dim(X) = \Omega(\sqrt{N})$ can be extracted from
Rudin \cite{Rudin60} (see also \cite{LLR95}
for a more accessible exposition).
Indyk \cite{Indyk2} presented a deterministic
polynomial-time construction with distortion $1+o(1)$ and
$\dim(X) \geq \frac{N}{\exp(O(\log \log N)^2)}.$

Another very interesting line of research pursued by various authors and
in quite different contexts is to achieve, in the terminology of
theoretical computer science, a {\em partial derandomization} of the
original (existential) results. The goal is to
come up with a ``constructive'' {\em discrete} probabilistic measure on
subspaces $X$ of ${\Bbb R}^N$ such that a random (with respect to this measure)
subspace still has low distortion almost surely, whereas the entropy of this
measure (that is, the number of truly random bits necessary to sample from
it) is also as low as possible.

Denoting by $\Ber{k,N}$ a random $k \times N$ sign matrix
(i.e. with i.i.d. Bernoulli $\pm 1$ entries),
one can extract from the paper \cite{kasin} by Kashin that $\ker(\Ber{k,N})$,
a subspace of codimension at most $k$ has, with high probability,
distortion $\sqrt{N/k} \cdot\mathrm{polylog}(N/k)$.
Schechtman \cite{Sch81} arrived at similar conclusions
for subspaces generated by rows of $\Ber{N-k,N}$.
Artstein-Avidan and Milman \cite{AM06} considered again
the model $\ker(\Ber{k,N})$ and derandomized this further
from $O(N^2)$ to $O(N\log N)$ bits of randomness.
We remark that
the pseudorandom generator approach of Indyk \cite{Indyk1} can be used
to efficiently construct such subspaces using $O(N \log^2 N)$ random
bits.
This was further improved to $O(N)$ bits by Lovett and Sodin \cite{LS07}.
Subsequent to our work, Guruswami, Lee, and Wigderson \cite{GLW08} used the construction approach from this paper to reduce the random bits to $O(N^\delta)$ for any $\delta > 0$ while achieving distortion $2^{O(1/\delta)}$.

\medskip

As far as deterministic constructions with $\dim(X) = \Omega(N)$ are concerned,
we are aware of only one result;  implicit in various papers (see e.g. \cite{DS89}) is a subspace
with $\dim(X) = N/2$ and distortion $O(N^{1/4})$.  For $\dim(X) \geq 3N/4$, say,
it appears that nothing non-trivial was shown prior to our work.

\smallskip
Our main result is as follows.

\begin{theorem}\label{thm:construction}
For every $\eta = \eta(N)$, there is an explicit,
deterministic polynomial-time construction of subspaces $X \subseteq \mathbb
R^N$ with $\dim(X) \geq (1-\eta)N$ and distortion $(\eta^{-1}\log \log N)^{O(\log\log
N)}$.
\end{theorem}
Like in \cite{kasin,AM06,LS07}, our space $X$ has the form $\ker(A_{k,N})$ for
a sign matrix $A_{k,N}$, but in our case this matrix is completely explicit (and, in particular, polynomial time computable). Its high-level overview is given in Section \ref{overview} below.

On
the other hand, if we allow ourselves a small number of random bits, then we
can slightly improve
the bound on distortion.

\begin{theorem} \label{thm:pseudorandom}
For every fixed $\eta >0$ there is a polynomial time algorithm using
$N^{1/\log\log N}$ random bits that almost surely produces a subspace $X
\subseteq \mathbb R^N$ with $\dim(X) \geq (1-\eta)N$ and distortion $(\log
N)^{O(1)}$.
\end{theorem}

\subsection{Proof techniques}

\subsubsection{Spreading subspaces} \label{initialize}

Low distortion of a section $X\subseteq {\Bbb R}^N$ intuitively means
that for every non-zero $x\in X$, a ``substantial'' portion of its
mass is spread over ``many'' coordinates, and we formalize this
intuition by introducing the concept of a {\em spread subspace}
(Definition \ref{def:spread}).  While this concept is
tightly related to distortion, it is far more convenient
to work with. In particular, using a simple spectral argument and
Kerdock codes~\cite{kerdock}, \cite[Chap. 15]{Mac-Sloane}, we
initialize our proof by presenting explicit subspaces with reasonably
good spreading properties.
These codes appeared also in the approach of Indyk \cite{Indyk2}, though
they were used in a dual capacity (i.e., as generator matrices instead of
check matrices).
In terms of distortion, however, this construction
can achieve at best $O(N^{1/4})$.

\subsubsection{The main construction}

The key contribution of our paper consists in exploiting the natural
analogy between low-distortion subspaces over the reals and
error-correcting codes over a finite alphabet.

Let $G = (\{1,2,\ldots,N\}, V_R, E)$ be a bipartite graph which is $d$-regular
on the right, and let $L \subseteq \mathbb R^d$ be any subspace.
Using the notation $\Gamma(j) \subseteq \{1,2,\ldots,N\}$ for the
neighbor set of a vertex $j \in V_R$,
we analyze the subspace
$$
X(G,L) = \{ x \in \mathbb R^N : x_{\Gamma(j)} \in L \textrm{ for every $j \in V_R$}\},
$$
where for $S \subseteq [N]$, $x_S \in \mathbb R^{|S|}$
represents the vector $x$ restricted to the coordinates
lying in $S$.  In other words, we impose local linear constraints (from $L$)
according to the structure of some bipartite graph $G$.
As Theorem \ref{thm:pushdown} shows, one can in particular analyze the spreading
properties of $X(G,L)$ in terms of those of $L$ and the expansion
properties of $G$.

\subsubsection{Putting it together: combinatorial overview} \label{overview}

Our final space $X$ will be of the form $X=\bigcap_{i=0}^{r-1} X(G_i,L_i)$ for suitably chosen $G_i, L_i$ (see
 the proof of Theorem \ref{thm:construction}). Combinatorially this simply means that we take $k_i\times N$ sign matrices $A_i$ such that $X(G_i,L_i)= \ker(A_i)$ and stack them on the top of one another to get our final matrix $A_{k,N}$. Moreover, every $A_i$ is a stack of $|V_R|$ copies of the sign matrix $A_i'$ with $\ker(A_i')=L_i$ in which every copy is padded with $(N-d)$ zero columns. The exact placement of these columns is governed by the graph $G$ that is chosen to satisfy certain expansion properties (Theorem \ref{thm:spectral}), and it is different in different copies.

And then we have one more level of recursion: Every $L_i$ has the form $X(G_i',L_i')$, where $G_i'$ again have certain (but this time different -- see Proposition \ref{prop:BKSSW}) expansion properties and $L_i'$ is our initial subspace (see Section \ref{initialize}).

\subsubsection{Connections to discrete codes}

Our approach is inspired by
{\em Low Density Parity Check Codes} (LDPC) introduced by Gallager
\cite{gallager-thesis}. They are particularly suited
to our purposes since, unlike most other explicit constructions in
coding theory, they exploit a {\em combinatorial} structure of the
parity check matrix and rely very little on the arithmetic of the
underlying finite field. Sipser and Spielman~\cite{SS96} showed that
one can achieve basically the same results (that is, simple and
elegant constructions of constant rate, constant relative minimal
distance codes) by considering adjacency matrices of sufficiently good
expanders instead of a {\em random} sparse matrix. These codes are
nowadays called {\em expander codes}. Using an idea due to
Tanner~\cite{tanner}, it was shown in \cite{SS96} (see also
\cite{zemor}) that even better constructions can be achieved by
replacing the parity check by a small (constant size) inner code.  Our
results demonstrate that analogous constructions work over the reals: If the
inner subspace $L$ has reasonably good spreading properties, then
the spreading properties of $X(G,L)$ are even better.
Upper bounds on distortion follow.

\subsection{Organization}
In Section \ref{sec:prelims}, we provide necessary background on
bipartite expander graphs and define spread subspaces.
In Section~\ref{sec:kerdock}, we
initialize our construction with an explicit subspace with reasonably
good spreading properties. In Section \ref{sec:boosting}, we describe
and analyze our main expander-based construction.
Finally, in Section \ref{sec:discussion}, we discuss why
improvements to our bounds may have to come from a source
other than better expander graphs.

\subsection{Relationship to compressed sensing.}
In \cite{DeVore07}, DeVore asks whether
probabilistically generated compressed sensing matrices can be
given by deterministic constructions.

The note \cite{KT07} makes the connection between distortion and
compressed sensing quite explicit.  If $M : \mathbb R^N \to \mathbb
R^{n}$ satisfies $\Delta(\ker(M)) \leq D$, then any vector $x \in
\mathbb R^N$ with $|\mathrm{supp}(x)| < \frac{N}{4 D^2}$ can be
uniquely recovered from its encoding $Mx$.  Moreover, given the
encoding $y = Mx$, the recovery can be performed efficiently by
solving the following convex optimization problem: $\min_{v \in
  \mathbb R^N} \|v\|_1$ subject to $Mv = y$.

In fact, something more general is shown.
Define, for $x \in \mathbb R^N$, the quantity
\begin{equation}\label{eq:approx}
\sigma_k(x)_1 = \min_{w \in \mathbb R^N : |\mathrm{supp}(w)| \leq k} \|x-w\|_1
\end{equation}
as the error of the best sparse approximation to $x$.
Then given $Mx$, the above algorithm recovers a vector $v \in \mathbb R^N$
such that $Mx = Mv$ and $\|x-v\|_2 \leq \frac{\sigma_k(x)_1}{\sqrt{k}},$
for $k = \Theta(N/D^2)$.  In other words, the recovery algorithm is {\em stable}
in the sense that it can also tolerate noise in the signal $x$, and
is able to perform approximate recovery even for signals which are
only approximately sparse.

Thus our results show the existence of a mapping
$M : \mathbb R^N \to \mathbb R^{o(N)}$, where $M$ is given
by an explicit matrix, and such that
any vector $x \in \mathbb R^N$ with $|\mathrm{supp}(x)| \leq \frac{N}{(\log N)^{C \log \log \log N}}$
can be efficiently recovered from $M x$ (the stable generalization
also holds, along the lines of \eqref{eq:approx}).
This yields the best-known explicit compressed sensing matrices for
this range of parameters (e.g. where $n \approx N/\mathrm{poly}(\log N)$).
Moreover, unlike probabilistic constructions, our matrices are quite sparse,
making compression (i.e., matrix-vector multiplication) and recovery
(via Basis Pursuit) more efficient.
For instance, when $n = N/2$, our matrices have only $N^{2-\varepsilon}$
non-zero entries for some $\varepsilon > 0$.
We refer to \cite{IndykCS} for explicit constructions that achieve a better
tradeoff for $n \approx N^{\delta}$, with $0 < \delta < 1$.  We remark
that the construction of \cite{IndykCS} is not stable in the sense
discussed above (and hence only works for actual sparse signals).

\section{Preliminaries}
\label{sec:prelims}

\subsection{Notation}

For two expressions $A,B$, we sometimes write $A \gtrsim B$ if $A =
\Omega(B)$, $A\lesssim B$ if $A=O(B)$, and we write $A \approx B$ if
$A=\Theta(B)$, that is $A \gtrsim B$ and $B \gtrsim A$.
For a positive integer $M$, $[M]$ denotes the set $\{1,2,\dots,M\}$.
The set of nonnegative integers is denoted by $\N$.
\subsection{Unbalanced bipartite expanders} \label{sec:expanders}

Our construction is based on unbalanced bipartite graphs with
non-trivial vertex expansion.

\begin{defn}
  A bipartite graph $G=(V_L,V_R,E)$ (with no multiple edges) is said to be
  an {\em $(N,n,D,d)$-right regular graph} if $|V_L| = N$, $|V_R| = n$, every
  vertex on the left hand size $V_L$ has degree {\em at most} $D$, and
  every vertex on the right hand side $V_R$ has degree {\em equal} to $d$.
\end{defn}

For a graph $G = (V,E)$ and a vertex $v \in V$, we denote by
$\Gamma_G(v)$ the {\em vertex neighborhood} $\{ u \in V \mid (v,u) \in
E\}$ of $v$. We denote by $d_v = |\Gamma_G(v)|$ the degree of a vertex
$v$.  The {\em neighborhood} of a subset $S \subseteq V$ is defined by
$\Gamma_G(S) = \bigcup_{v \in S} \Gamma_G(v)$. When the graph $G$ is
clear from the context, we may omit the subscript $G$ and denote the
neighborhoods as just $\Gamma(v)$ and $\Gamma(S)$.

\begin{defn}[Expansion profile]
  The {\em expansion profile} of a bipartite graph $G=(V_L,V_R,E)$ is
  the function $\Lambda_G : (0,|V_L|] \to \N$ defined by
$$
\Lambda_G(m) = \min\left\{|\Gamma_G(S)| : S \subseteq V_L, |S| \geq m \right\}.
$$
\end{defn}
Note that $\Lambda_G(m) = \min_{v \in V_L} d_v$ for $0 < m \le 1$.

For our work, we need {\em unbalanced} bipartite graphs with expansion
from the larger side to the smaller side.  Our results are based on
two known explicit constructions of such graphs.  The first one is to
take the edge-vertex incidence graph of a non-bipartite {\em spectral
  expander}\footnote{That is, a regular graph with a large gap between
  the largest and second largest eigenvalue of its adjacency matrix.}
such as a Ramanujan graph. These were also the graphs used in the work
on expander codes~\cite{SS96,zemor}. The second construction of
expanders is based on a suggestion due to Avi Wigderson. It uses a
result of Barak, et. al. \cite{BKSSW} based on sum-product estimates
in finite fields; see \cite[\S 2.8]{TVbook} for background on such
estimates.

For our purposes, it is also convenient (but not strictly necessary)
to have bipartite
graphs that are regular on the right. We begin by describing a simple
method to achieve right-regularity with minimal impact on the
expansion and degree parameters, and then turn to stating the precise
statements about the two expander constructions we will make use of in
Section~\ref{sec:boosting} to construct our explicit subspaces.

\subsubsection{Right-regularization}

\begin{lemma}
\label{lem:right-regular}
  Given a graph $H = (V_L,V_R,E)$ with $|V_L| = N$,
  $|V_R| = n$, that is left-regular with each vertex in $V_L$ having
  degree $D$, one can construct in $O(ND)$ time an $(N,n',2D,d)$-right
  regular graph $G$ with $n' \le 2n$ and $d = \lceil \frac{ND}{n} \rceil$
  such that the expansion profiles satisfy
  $\Lambda_G(m) \geq \Lambda_H(m)$ for all $m>0$.
\end{lemma}
\begin{proof}
  Let $d_{\rm av} = ND/n$ be the average right degree of the graph
  $H$ and let $d = \lceil d_{\rm av} \rceil$. Split each vertex $v \in
  V_R$ of degree $d_v$ into $\lfloor d_v/d \rfloor$ vertices of degree
  $d$ each, and if $d_v \mod d > 0$, a ``remainder'' vertex of degree
  $r_v = d_v \mod d$. Distribute the $d_v$ edges incident to $v$ to
  these split vertices in an arbitrary way. The number of newly
  introduced vertices is at most $\sum_{v \in V_R} d_v/d = n d_{\rm av}/d
  \le n$, so the number $n'$ of right-side vertices in the new graph
  satisfies $n' \le 2n$.

  All vertices except the at most $n$ ``remainder'' vertices now have degree
  exactly $d$. For each $v \in V_R$, add $d - r_v$ edges to the corresponding
  remainder vertex (if one exists). Since this step adds at most $(d-1)
  n \le d_{\rm av} n = N D$ edges, it is possible to distribute these
  edges in such a way that no vertex in $V_L$ is incident on more than
  $D$ of the new edges. Therefore, the maximum left-degree of the new
  graph is at most $2D$.

The claim about expansion is obvious --- just ignore the newly added
edges, and the splitting of vertices can only improve the vertex expansion.
\end{proof}

\subsubsection{Spectral expanders}

The next theorem converts non-bipartite expanders
to unbalanced bipartite expanders via the usual
edge-vertex incidence construction.

\begin{theorem} \label{thm:spectralfirst}
For every $d \geq 5$ and $N \geq d$, there exists an explicit
$(N'=\Theta(N),~ n,2, \Theta(d))$-right regular graph $G$
whose
expansion profile satisfies $\Lambda_G(m) \geq \min\left\{\frac{m}{2\sqrt{d}},\frac{\sqrt{2mN'}}{d}\right\}$.
\end{theorem}

\begin{proof}
  Let $p,q$ be any two primes which are both congruent to 1 modulo 4.
  Then there exists an explicit $(p+1)$-regular graph $Y = (V, F)$
  with $\frac{q(q^2-1)}{4} \leq |V| \leq \frac{q(q^2-1)}{2}$ and such
  that $\lambda_2 = \lambda_2(Y) \leq 2 \sqrt{p}$, where
  $\lambda_2(Y)$ is the second largest eigenvalue (in absolute value)
  of the adjacency matrix of $Y$~\cite{LPS88}.  (See \cite[\S
  2]{SLW06} for a discussion of explicit constructions of expander
  graphs.)

Letting $n = |V|$, we define a $(\frac{(p+1)n}{2}, n, 2, p+1)$-right
regular bipartite graph $G = (V_L, V_R, E)$ where $V_L = F$,
$V_R = V$, and $(e,v) \in E$ if $v$ is an endpoint of $e \in F$.
To analyze the expansion properties of $G$, we use
the following lemma of Alon and Chung \cite{AC88}.

\begin{lemma}
\label{lem:alon-chung}
If $Y$ is any $d$-regular graph on $n$ vertices
with second eigenvalue $\lambda_2$, then
the induced subgraph on any set of $\gamma n$ vertices
in $Y$
has at most
$$
\of{\gamma^2 + \gamma \frac{\lambda_2}{d}}\frac{dn}{2}
$$
edges.
\end{lemma}

In particular, if $S \subseteq V_L$ satisfies $|S| \geq \gamma^2 (p+1) n$
and $|S| \geq 2 \gamma n \sqrt{p+1}$, then $|\Gamma_G(S)| \geq \gamma n$.
Stated different, for any $S \subseteq V_L$, we have
$$
|\Gamma_G(S)| \geq \frac{\min\left\{2\sqrt{|S| n},|S| \right\}}{2\sqrt{p+1}}
$$
Setting $N' = \frac{(p+1)n}{2}$, we see that
$\Lambda_G(m) \geq \min\left\{\frac{m}{2\sqrt{d}},\frac{\sqrt{2 N'm}}{d}\right\}$.

Now given parameters $d \geq 5$ and $N \geq d$, let $p$ be the largest
prime satisfying $p+1 \leq d$ and $p \equiv 1\,(\bmod\,4)$,
and let $q$ be the smallest prime satisfying $\frac{q(q^2-1)(p+1)}{8} \geq N$
and $q \equiv 1\,(\bmod\,4)$.
The theorem follows by noting that for all integers $m \geq 3$, there
exists a prime $p \in [m,2m]$ which is congruent to 1 modulo 4 (see
\cite{ErdosBertrand}).
\end{proof}

The expanders of Theorem~\ref{thm:spectralfirst} are already right-regular
but they have one drawback; we cannot fully control the number of
left-side vertices $N$. Fortunately, this can be easily circumvented with
the same Lemma~\ref{lem:right-regular}.

\begin{theorem}
\label{thm:spectral} For every $d \geq 5$ and $N \geq d$, there exists an
explicit $(N,n,4, \Theta(d))$-right regular graph $G$ which satisfies
$\Lambda_G(m) \geq \min\left\{\frac{m}{2\sqrt{d}},\frac{\sqrt{2Nm}}{d}\right\}$.
\end{theorem}
\begin{proof}
  Apply Theorem \ref{thm:spectralfirst} to get a graph with $N' \geq
  N$, $N'\approx N$ vertices on the left, then remove an arbitrary subset
  of $N-N'$ vertices from the left hand side.  This doesn't affect the
  expansion properties, but it destroys right-regularity.  Apply
  Lemma~\ref{lem:right-regular} to correct this.
\end{proof}

\subsubsection{Sum-product expanders}

In this section, $p$ will denote a prime, and $\F_p$ the finite field
with $p$ elements.  The following result is implicit in \cite[\S 4]{BKSSW},
and is based on a key ``sum-product'' lemma (Lemma 3.1) from
\cite{BIW}, which is itself a statistical version of the
sum-product theorems of Bourgain, Katz, and Tao \cite{BKT03},
and Bourgain and Konyagin \cite{BK03} for finite fields.

\begin{prop}
\label{prop:BKSSW}
  There exists an absolute constant $\xi_0 > 0$ such that for all
  primes $p$ the following holds. Consider the bipartite graph $G_p =
  (\F_p^3, [4] \times \F_p, E)$ where a left vertex $(a,b,c) \in
  \F_p^3$ is adjacent to $(1,a)$, $(2,b)$, $(3,c)$, and $(4,a \cdot b
  + c)$ on the right. Then $\Lambda_{G_p}(m)\geq \min\left\{ p^{0.9},
  m^{1/3+\xi_0}\right\}$.
\end{prop}

Note that trivially $|\Gamma_{G_p}(S)| \ge |S|^{1/3}$, and the above
states that not-too-large sets $S$ expand by a sizeable amount more
than the trivial bound. Using the above construction, we can now prove
the following.

\begin{theorem}
  \label{thm:SP-regular} For all integers $N \geq 1$, there
  is an explicit construction of an $(N,n,8,\Theta(N^{2/3}))$-right
  regular graph $G$ which
  satisfies
  $$\Lambda_G(m) \geq \min\left\{\tfrac18 n^{0.9}, m^{1/3+\xi_0} \right\}.$$
(Here $\xi_0$ is
the absolute constant from Proposition~\ref{prop:BKSSW}.)
\end{theorem}
\begin{proof}
  Let $p$ be the smallest prime such that $p^3 \ge N$; note that $N^{1/3}\leq p
  \le 2 N^{1/3}$. Construct the graph $G_p$, and a subgraph $H$ of
  $G_p$ by deleting an arbitrary $p^3 - N$ vertices on the left. Thus
  $H$ has $N$ vertices on left, $4p$ vertices on the right, is
  left-regular with degree $4$ and satisfies, by
  Proposition~\ref{prop:BKSSW}, $\Lambda_H(m) \geq \min\left\{ p^{0.9} ,
    m^{1/3+\xi_0} \right\}$. Applying the transformation
  of Lemma~\ref{lem:right-regular} to $H$, we get an $(N,n,8,d)$-right
  regular graph with $d = \lceil
  \frac{4N}{4p}\rceil \approx N^{2/3}$ and
  with the same expansion property.
\end{proof}

\subsection{Distortion and spreading}
\label{sec:spread}

For a vector $x\in \R^N$ and a subset $S\subseteq [N]$ of coordinates,
we denote by $x_S\in {\mathbb R}^{|S|}$ the projection of $x$ onto the
coordinates in $S$.  We abbreviate the complementary set of
coordinates $[N]\setminus S$ to $\bar S$.

\begin{defn}[Distortion of a subspace]
For a subspace $X \subseteq \R^N$, we define
$$
\Delta(X) = \sup_{x \in X \atop {x \neq 0}} \frac{\sqrt{N} \|x\|_2}{\|x\|_1}.
$$
\end{defn}

As we already noted in the introduction, instead of distortion it turns out
to be more convenient to work with the following notion.

\begin{defn} \label{def:spread}
A subspace $X\subseteq \R^N$ is {\em $(t,\epsilon)$-spread} if for every
$x\in X$ and every $S\subseteq [N]$ with $|S|\leq t$, we have
$$
\|x_{\bar S}\|_2 \geq\epsilon\cdot\|x\|_2.
$$
\end{defn}

Let us begin with relating these two notions.

\begin{lemma} \label{lem:spread_vs_distortion}
Suppose $X \subseteq \mathbb R^N$.
\begin{tenumerate}
\item \label{1} If $X$ is $(t,\epsilon)$-spread then
$$
\Delta(X) \leq \sqrt{\frac Nt}\cdot\epsilon^{-2};
$$

\item \label{2} conversely, $X$
is $\of{\frac N{2\Delta(X)^2},\ \frac 1{4\Delta(X)}}$-spread.
\end{tenumerate}
\end{lemma}

\begin{proof}
\ref{1}. Fix $x\in X$; we need to prove that
\begin{equation} \label{eq:distortion}
\|x\|_1\geq \sqrt t \epsilon^2\|x\|_2.
\end{equation}
W.l.o.g. assume that $\|x\|_2=1$ and that $|x_1|\geq |x_2|\geq\ldots
\geq |x_N|$. Applying Definition \ref{def:spread}, we know that
$\|x_{[t+1..N]}\|_2\geq\epsilon$. On the other hand, $\sum_{i=1}^t|x_i|^2\leq 1$,
therefore $|x_t|\leq \frac 1{\sqrt t}$ and thus $\|x_{[t+1..N]}\|_\infty\leq
\frac 1{\sqrt t}$. And now we get (\ref{eq:distortion}) by the calculation
$$
\|x\|_1\geq \|x_{[t+1..N]}\|_1 \geq \frac{\|x_{[t+1..N]}\|_2^2}{\|x_{[t+1..N]}\|_\infty}
\geq \sqrt t\epsilon^2.
$$

\medskip
\ref{2}. Let $t=\frac{N}{2\Delta(X)^2}$. Fix again $x\in X$ with $\|x\|_2=1$ and
$S\subseteq [N]$ with $|S|\leq t$. By the bound on distortion, $\|x\|_1\geq
\frac{\sqrt N}{\Delta(X)}$. On the other hand,
$$
\|x_S\|_1 \leq \sqrt t\cdot \|x_S\|_2\leq \sqrt t= \frac{\sqrt{N/2}}{\Delta(X)},
$$
hence $\|x_{\bar S}\|_1 = \|x\|_1-\|x_S\|_1\geq \frac{\sqrt N}{4\Delta(X)}$ and $\|x_{\bar S}\|_2\geq
\frac{\|x_{\bar S}\|_1}{\sqrt N}\geq \frac 1{4\Delta(X)}$.
\end{proof}

Next, we note spreading properties of random subspaces (they will be
needed only in the proof of Theorem \ref{thm:pseudorandom}).  The
following theorem is due to Kashin \cite{kasin}, with the optimal
bound essentially obtained by Garnaev and Gluskin
\cite{garnaev-gluskin}.  We note that such a theorem now follows from
standard tools in asymptotic convex geometry, given the entropy bounds
of Sch\"utt \cite{schutt84} (see, e.g.  Lemma B in \cite{LS07}).

\begin{theorem} \label{thm:random_distortion}
If $A$ is a uniformly random $k \times N$ sign matrix, then with
probability $1 - o(1)$,
$$\Delta(\ker(A)) \lesssim \sqrt{\frac{N}{k} \log \of{\frac{N}{k}}}.$$
\end{theorem}

Combining Theorem \ref{thm:random_distortion} with Lemma
\ref{lem:spread_vs_distortion}(\ref{2}, we get:

\begin{theorem} \label{thm:random_spread}
If $A$ is a uniformly random $k \times N$ sign matrix, then with probability
$1 - o(1)$, $\ker(A)$ is a $\of{\Omega\of{\frac k{\log (N/k)}},
\Omega\of{\sqrt{\frac k{N\log(N/k)}}}}$-spread subspace.
\end{theorem}

Finally, we introduce a ``relative'' version of Definition \ref{def:spread}. It
is somewhat less intuitive, but very convenient to
work with.

\begin{defn} \label{def:relative}
A subspace $X\subseteq \R^N$ is {\em $(t,T,\epsilon)$-spread} ($t\leq T$)
if for every $x\in X$,
$$
\min_{S\subseteq [N]\atop |S|\leq T}\|x_{\bar S}\|_2 \geq \epsilon\cdot
\min_{S\subseteq [N]\atop |S|\leq t}\|x_{\bar S}\|_2.
$$
\end{defn}

Note that $X$ is $(t,\epsilon)$-spread if and only if it is
$(0,t,\epsilon)$-spread, if and only if
it is $(1/2,t,\epsilon)$-spread. (Note that $t,T$
are not restricted to integers in our definitions.)
One obvious advantage of Definition \ref{def:relative} is that it
allows us to break the task of constructing well-spread subspaces into
pieces.

\begin{lemma} \label{lem:iterate}
Let $X_1,\ldots,X_r\subseteq \R^N$ be linear subspaces, and assume that $X_i$
is $(t_{i-1},t_{i},\epsilon_i)$-spread, where $t_0\leq t_1\leq\cdots\leq
t_r$. Then $\bigcap_{i=1}^r X_i$ is
$(t_0,t_r,\prod_{i=1}^r\epsilon_i)$-spread.
\end{lemma}

\begin{proof}
Obvious.
\end{proof}

\section{An explicit weakly-spread subspace}
\label{sec:kerdock}
Now our goal can be stated as finding an explicit
construction that gets as close as possible to the probabilistic bound
of Theorem \ref{thm:random_spread}. In this section we perform a
(relatively simple) ``initialization'' step; the boosting argument
(which is the most essential contribution of our paper) is deferred to
Section \ref{sec:boosting}.  Below, for a matrix $A$, we denote by
$\|A\|$ its operator norm, defined as $\sup_{x \neq 0}
\frac{\|Ax\|_2}{\|x\|_2}$.

\begin{lemma}
\label{lem:spread-kernel} Let $A$ be any $k\times d$ matrix whose columns
$a_1,\ldots,a_d\in \R^k$ have $\ell_2$-norm 1, and, moreover, for any $1\leq
i<j\leq d$, $|\langle a_i,a_j\rangle|\leq\tau$. Then $\ker(A)$ is $\of{\frac
1{2\tau}, \frac 1{2\|A\|}}$-spread.
\end{lemma}
\begin{proof}
Fix $x \in \ker(A)$ and let $S \subseteq [d]$ be any subset with $t
= |S| \leq \frac{1}{2\tau}$.
Let $A_S$ be the $k \times t$ matrix which arises by restricting $A$
to the columns indexed by $S$, and let $\Phi = A_S^{\mathsf{T}}
A_S$.  Then $\Phi$ is the $t \times t$ matrix whose entries are
$\langle a_i, a_j\rangle$ for $i,j \in S$, therefore we can write
$\Phi = I + \Phi'$ where every entry of $\Phi'$ is bounded in
magnitude by $\tau$.  It follows that all the eigenvalues of $\Phi$ lie
in the range $\left[1-t\tau,1+t\tau\right]$. We conclude, in particular,
that $\|A_S\, y\|_2^2 \geq (1-t\tau) \|y\|_2^2 \geq \frac12 \|y\|_2^2$
for every $y \in \mathbb R^t$.

Let $A_{\bar S}$ be the restriction of $A$ to the columns in the
complement of $S$. Since $x \in \ker(A)$, we have
$$
0 = Ax = A_S x_S + A_{\bar S} x_{\bar S}
$$
so that $$\|A_{\bar S} x_{\bar S}\|_2 = \|A_S x_S\|_2 \geq
\frac{1}{\sqrt{2}} \|x_S\|_2.$$ Since
$\|A_{\bar S} x_{\bar S}\|_2 \leq \|A\| \cdot \|x_{\bar S}\|_2,$ it
follows that $\|x_S\|_2 \leq \sqrt{2} \|A\|\cdot \|x_{\bar S}\|_2$.
Since $\|A\|\geq 1$, this implies $\|x_{\bar S}\|_2 \geq \frac{\|x\|_2}{2\|A\|}$.
\end{proof}

We now obtain matrices with small operator norm and near-orthogonal
columns from explicit constructions of Kerdock codes.

\begin{prop}
\label{prop:kerdock}
For all positive integers $d,k$ where $k$ is a power of $4$ satisfying $k \le d \le k^2/2$,
there exists an explicit $k \times d$ matrix $A$ with the following properties.
\begin{enumerate}
\itemsep=0ex
\item Every entry of $A$ is either $\pm 1/\sqrt{k}$, and thus the
  columns $a_1,a_2,\dots,a_d \in \R^k$ of $A$ all have $\ell_2$-norm
  $1$,
\item For all $1 \le i < j \le d$, $|\langle a_i, a_j \rangle| \le 1/\sqrt{k}$, and
\item $\|A\| \le \sqrt{\left\lceil \frac{d}{k} \right\rceil}$.
\end{enumerate}
\end{prop}
\begin{proof}
  The proof is based on a construction of mutually unbiased bases over
  the reals using Kerdock codes~\cite{kerdock,cameron-seidel}.  First,
  let us recall that for $k$ a power of $2$, the Hadamard code of
  length $k$ is a subspace of $\F_2^k$ of size $k$ containing the $k$
  linear functions $L_a : \F_2^{\log_2 k} \rightarrow \F_2$, where for
  $a,x \in \F_2^{\log_2 k}$, $L_a(x) = a \cdot x$ (computed over
  $\F_2$).  A Kerdock code is the union of a Hadamard code $H
  \subseteq \F_2^k$ and a collection of its cosets $\{ f+H \mid f \in
  {\cal F}\}$, where ${\cal F}$ is a set of quadratic bent functions
  with the property that for all $f \neq g \in {\cal F}$, the function
  $f+g$ is also bent.\footnote{A function $f : \F_2^a \rightarrow
    \F_2$ for $a$ even is said to be {\em bent} if it is maximally far
    from all linear functions, or equivalently if all its Fourier
    coefficients have absolute value $1/2^{a/2}$.}%

  When $k$ is a power of $4$, it is known (see \cite{kerdock} and also
  \cite[Chap. 15, Sec. 5]{Mac-Sloane}) that one can construct an
  explicit set ${\cal F}$ of $(\frac{k}{2}-1)$ such bent functions. (A
  simpler construction of $(\sqrt{k}-1)$ such quadratic functions
  appears in \cite{cameron-seidel}.)  The cosets of these functions
  together with the Hadamard code (the trivial coset) give an explicit
  Kerdock code of length $k$ that has $k^2/2$ codewords. Interpreting
  binary vectors of length $k$ as unit vectors with $\pm 1/\sqrt{k}$
  entries, every coset of the Hadamard code gives an orthonormal basis
  of $\R^k$. The $k/2$ cosets comprising the Kerdock code thus yield $k/2$
  orthonormal bases $B_1,B_2,\dots,B_{k/2}$ of $\R^k$ with the
  property that for every pair $\{v,w\}$ of vectors in different bases,
  one has $|\langle v,w \rangle| = 1/\sqrt{k}$. (Such bases are called
  mutually unbiased bases.)

  For any $d$, $k \le d \le k^2/2$, write $d = qk+r$ where $0 \le r <
  k$. We construct our $k \times d$ matrix $A$ to consist of $[ B_1
  \dots B_q]$ followed by, in the case of $r > 0$, any $r$ columns of $B_{q+1}$. The first two
  properties of $A$ are immediate from the property of the bases
  $B_i$. To bound the operator norm, note that being an orthonormal
  basis, $\|B_i\| = 1$ for each $i$. A simple application of
  Cauchy-Schwartz then shows that $\|A\| \le \sqrt{\lceil d/k \rceil}$.
\end{proof}

Plugging in the matrices guaranteed by Proposition~\ref{prop:kerdock}
into Lemma~\ref{lem:spread-kernel}, we can conclude the following.
\begin{theorem}
\label{thm:local-subspace}
  For every integer $k$ that is a power of $4$ and every integer $d$ such
  that
  \begin{equation} \label{eq:restrictions}
  k \le d \le k^2/2,
  \end{equation}
  there exists a $\of{\frac{\sqrt{k}}{2} , \frac{1}{4} \sqrt{\frac kd}
  }$-spread subspace $L\subseteq \R^d$ with $\codim(L)\leq k$, specified as
  the kernel of an explicit $k \times d$ sign matrix.
\end{theorem}
These subspaces will be used as ``inner'' subspaces in
an expander-based construction (Theorem~\ref{thm:guv_boosting}) to get
a subspace with even better spreading properties.

\section{Boosting spreading properties via
  expanders} \label{sec:boosting}

\subsection{The Tanner construction}

\begin{defn}[Subspaces from bipartite graphs]
  Given a bipartite graph $G = (\{1,2,\ldots,N\},V_R,E)$
  such that every vertex in $V_R$ has degree $d$, and a subspace $L
  \subseteq \R^d$, we define the subspace $X = X(G,L) \subseteq \R^N$ by
\begin{equation}
\label{eq:subspace-defn}
X(G,L) = \{x  \in \R^N \mid x_{\Gamma_G(j)} \in L \mbox{ for every $j \in V_R$} \} \ .
\end{equation}
\end{defn}
The following claim is straightforward.

\begin{claim} \label{11}
If $n = |V_R|$, then $\codim(X(G,L))\leq \codim(L)n$, that is
$\dim(X(G,L)) \geq N - (d-\dim(L))n$.
\end{claim}

\begin{remark}[Tanner's code construction]
  Our construction is a continuous analog of
  Tanner's construction of error-correcting
  codes~\cite{tanner}. Tanner constructed codes by
  identifying the vertices on one side of a bipartite graph with the
  bits of the code and identifying the other side with constraints.
  He analyzed the performance of such codes by examining the girth
  of the bipartite graph. Sipser and Spielman~\cite{SS96}
  showed that graph expansion plays a key role in the quality
  of such codes, and gave a linear time decoding algorithm to correct a
  constant fraction of errors.  In the coding world,
  the special case when $L$ is the
  $(d-1)$-dimensional subspace $\{y \in \R^d \mid \sum_{\ell = 1}^d
  y_{\ell} = 0 \}$ corresponds to the low-density parity check codes
  of Gallager~\cite{gallager-thesis}. In this
  case, the subspace is specified as the kernel of the bipartite
  adjacency matrix of $G$.
\end{remark}

\subsection{The spread-boosting theorem}

We now show how
to improve spreading
properties using the above construction.

\begin{theorem} \label{thm:pushdown}
Let $G$ be an $(N,n,D,d)$-graph with expansion profile $\Lambda_G(\cdot)$,
and let $L\subseteq \R^d$ be a $(t,\epsilon)$-spread subspace. Then for every
$T_0$, $0 < T_0 \leq N$, $X(G,L)$ is $\of{T_0, \frac{t}{D} \Lambda_G(T_0), \frac{\epsilon}{\sqrt{2D}}}$-spread.
\end{theorem}
\begin{proof}
Fix $x\in X(G,L)$ with $\|x\|_2=1$.
Fix also
$S\subseteq [N]$ with $|S|\leq T$, where $T = \frac{t}{D} \Lambda_G(T_0)$. We then need to prove that
\begin{equation} \label{eq:min_bound}
\|x_{\bar S}\|_2 \geq \frac{\epsilon}{\sqrt{2D}}
\min_{|B| \leq T_0} \|x_{\bar B}\|_2.
\end{equation}
Let
$$
Q=\left\{ j\in [n] : |\Gamma(j)\cap S|>t \right\},
$$
and
$$
B=\left\{i\in S : \Gamma(i)\subseteq Q\right\}.
$$
Then
$$
t|Q| <  E(S,\Gamma(S))\leq D|S|\leq DT,
$$
therefore
$$
|Q| < \frac{DT}{t} = \Lambda_G(T_0)  \ .
$$
On the other hand, we have
$|Q| \geq |\Gamma(B)|$, and hence $|\Gamma(B)| < \Lambda_G(T_0)$.
By the definition of the expansion profile, this implies that $|B| < T_0$,
and therefore (see \eqref{eq:min_bound}) we are only left to show that
\begin{equation}
\label{eq:required} \|x_{\bar S}\|_2 \geq \frac{\epsilon}{\sqrt{2D}}
\cdot\|x_{\bar B}\|_2
\end{equation}
for our particular $B$.

Note first that
\begin{equation} \label{eq:bound1}
\|x_{\bar B}\|_2^2 = \|x_{\bar S}\|_2^2+\|x_{S\setminus B}\|_2^2.
\end{equation}
Next, since every vertex in $S \setminus B$ has at least one neighbor in
$\Gamma(S)\setminus Q$, we have
\begin{equation}
\sum_{j\in\Gamma(S)\setminus Q}\|x_{\Gamma(j)}\|_2^2\geq \|x_{S\setminus
B}\|_2^2.
\end{equation}
Since $x\in X(G,L)$, $L$ is $(t,\epsilon)$-spread, and
$|\Gamma(j)\cap S|\leq t$ for any $j\in \Gamma(S)\setminus Q$,
\begin{equation}
\sum_{j\in\Gamma(S)\setminus Q}\|x_{\Gamma(j)\setminus S}\|_2^2 \geq
\epsilon^2\cdot \sum_{j\in\Gamma(S)\setminus Q}\|x_{\Gamma(j)}\|_2^2.
\end{equation}
Finally,
\begin{equation} \label{eq:bound4}
\sum_{j\in\Gamma(S)\setminus Q}\|x_{\Gamma(j)\setminus S}\|_2^2 \leq
\sum_{j\in [n]}\|x_{\Gamma(j)\setminus S}\|_2^2\leq D\cdot \|x_{\bar S}\|_2^2.
\end{equation}

\eqref{eq:bound1}-\eqref{eq:bound4} imply
$$
\|x_{\bar S}\|_2^2 \geq \frac{\epsilon^2}{D} (\|x_{\bar B}\|_2^2-\|x_{\bar
S}\|_2^2).
$$
Since $\epsilon\leq 1$ and $D\geq 1$, \eqref{eq:required} (and hence Theorem
\ref{thm:pushdown}) follows.
\end{proof}

\subsection{Putting things together}

In this section we assemble the proofs of Theorems
\ref{thm:construction} and \ref{thm:pseudorandom} from the already
available blocks (which are Theorems \ref{thm:SP-regular},
\ref{thm:spectral}, \ref{thm:random_spread}, \ref{thm:local-subspace}
and \ref{thm:pushdown}). Let us first see what we can do using
expanders from Theorem \ref{thm:SP-regular}.

\subsubsection{First step: Boosting with sum-product expanders}
\label{sec:boost}

The main difference between the explicit construction of Theorem
\ref{thm:local-subspace} and the probabilistic result (Theorem
\ref{thm:random_spread}) is the order of magnitude of $t$ (the parameter from
Definition \ref{def:spread}). As we will see in the next section, this
difference is very principal, and our first goal is to {\em somewhat} close
the gap with an {\em explicit} construction.

\begin{theorem} \label{thm:guv_boosting}
Fix an arbitrary constant $\beta_0<\min\left\{0.08,\frac 38\xi_0\right\}$, where
$\xi_0$ is the constant from Theorem \ref{thm:SP-regular}. Then
for all sufficiently large $N \in \N$ and $\eta\geq N^{-2\beta_0/3}$ there exists an
explicit subspace $X\subseteq \R^N$ with $\codim(X)\leq\eta N$ which is
$(N^{\frac12+\beta_0},\eta^{O(1)})$-spread.
\end{theorem}

\begin{proof}
In everything that follows, we assume that $N$ is sufficiently large.
The desired $X$ will be of the form $X(G,L)$, where $G$ is supplied by
Theorem \ref{thm:SP-regular}, and $L$ by Theorem
\ref{thm:local-subspace}. More specifically,
let $G$ be the explicit $(N,n,8,d)$-right regular graph from Theorem \ref{thm:SP-regular}
with $d \approx N^{2/3}$ (and hence $n\approx N^{1/3}$).  Using Theorem \ref{thm:SP-regular},
one can check that for $m \leq N^{\frac12 + \beta_0}$,
we have
\begin{equation}\label{eq:profile}
\Lambda_G(m) \geq m d^{\beta_0 - \frac12}\ .
\end{equation}
Indeed, since $n\approx N^{1/3}$ and $d\approx N^{2/3}$, the inequality
$\frac 18n^{0.9}\geq md^{\beta_0-\frac 12}$ follows (for large $N$) from $\beta_0<0.08$, and the
inequality $m^{\frac 13+\xi_0}\geq md^{\beta_0-\frac 12}$ follows from
$\beta_0<\frac 38\xi_0$.

By our assumption $\eta\geq N^{-2\beta_0/3}\geq N^{-0.1}$, along with $d \approx
N^{2/3}$, we observe that $d\leq o\of{\eta d}^2$.  Hence
(cf. the statement of Theorem \ref{thm:local-subspace}), we can find
$k\leq \frac{\eta d}8$, $k\approx \eta d$ that is a power of 4 and
also satisfies the restrictions \eqref{eq:restrictions}. Let $L$ be an
explicit $\of{\Omega\of{\sqrt{\eta d}},
  \Omega\of{\sqrt{\eta}}}$-spread subspace guaranteed by Theorem
\ref{thm:local-subspace}.

The bound on codimension of $X(G,L)$ is obvious:
$\codim(X(G,L))\leq kn\leq \frac{\eta d
n}{8}\leq \eta N$.

For analyzing spreading properties of $X(G,L)$,
we observe that $\eta\geq N^{-2\beta_0/3}$
implies $\eta d \gtrsim d^{1-\beta_0}$,
hence $L$ is $(\Omega(d^{\frac12-\frac{\beta_0}{2}}), \eta^{O(1)})$-spread.
By Theorem \ref{thm:pushdown} and \eqref{eq:profile},
for every $T \leq N^{\frac12+\beta_0}$, we know that
$X(G,L)$ is $(T, \Omega(d^{\frac{\beta_0}{2}}) T, \eta^{O(1)})$-spread
In particular, for such $T$, $X(G,L)$ is $(T, N^{\Omega(1)} T, \eta^{O(1)})$-spread.

Applying Lemma
\ref{lem:iterate} with the same spaces $X_1:=\cdots:=X_r:=X(G,L)$ and
suitably large constant $r \approx 1/\beta_0 = O(1)$, we conclude that $X(G,L)$ is $\of{\frac12, N^{\frac12 + \beta_0},
\eta^{O(1)}}$-spread, completing the proof.
\end{proof}

\remove{
\begin{remark}
By letting  $c\rightarrow 0$ in Theorem 2.6, we could further push the value
of the parameter $t$ from $N^{2/3}$ to $N^{1-o(1)}$, and this alone would
give explicit subspaces $X$ with $\Delta(X)\leq N^{o(1)}$. But in the next
section we will be able to achieve much better results using spectral
expanders from Theorem \ref{thm:spectral}, so we do not elaborate on the
details.
\end{remark}
}

\subsubsection{Second step: Handling large sets based on spectral expanders}

The sum-product expanders of Theorem \ref{thm:SP-regular} behave
poorly for very large sets (i.e., as $m \to N$, the lower bound on
$\Lambda_G(m)$ becomes constant from some point).  The spectral
expanders of Theorem \ref{thm:spectral} behave poorly for small sets,
but their expansion still improves as $m \to N$.  In this section, we
finish the proofs of Theorems \ref{thm:construction} and
\ref{thm:pseudorandom} by exploring strong sides of both
constructions. We begin with Theorem \ref{thm:pseudorandom} as it is
conceptually simpler (we need only spectral expanders, do not rely on
Theorem \ref{thm:guv_boosting}, and still use only one fixed space
$X(G,L)$).

\bigskip
\noindent
{\bf Proof of Theorem \ref{thm:pseudorandom}.} By Theorem
\ref{thm:spectral} there exists an explicit $(N,n,4,d)$-right regular graph $G$
with
\begin{equation}\label{eq:another_bound_on_d}
  N^{\Omega\of{\frac 1{\log\log N}}} \leq d \leq N^{\frac 1{2\log\log N}}
\end{equation}
which has
$\Lambda_G(m) \geq \min\left\{ \frac m{2\sqrt
d},\ \frac {\sqrt{2Nm}}{d}\right\}$. Let $k=\lfloor
\frac{\eta}4d\rfloor$; our desired (probabilistic) space is then
$X(G,\ker(A))$, where $A$ is a uniformly random $k\times d$ sign matrix (due
to the upper bound in \eqref{eq:another_bound_on_d}, this uses at most
$d^2\leq N^{\frac 1{\log\log N}}$ random bits).
Recalling that $\eta>0$ is an absolute constant, by Theorem
\ref{thm:random_spread} $\ker(A)$ is an $(\Omega(d),\Omega(1))$-spread
subspace almost surely.

The bound on codimension is again simple: $\codim(X(G,\ker(A)))\leq
kn\leq\eta N$.

For analyzing spreading properties of $X$, let $m_0=8N/d$ (which is
the ``critical'' point where $\frac{m_0}{2\sqrt{d}} =\frac{\sqrt{2Nm_0}}{d}$.)
Then Theorem \ref{thm:pushdown}
says that $X(G,L)$ is
\begin{enumerate}
\itemsep=0ex
\item[{a.}] $\of{T,\Omega(\sqrt{d}) T,\Omega(1)}$-spread subspace for $T\leq m_0$, and
\item[{b.}] $\of{T, \Omega(\sqrt{NT}), \Omega(1)}$-spread
subspace for $m_0\leq T\leq N$.
\end{enumerate}

And now we are once more applying Lemma \ref{lem:iterate} with
$X_1:=X_2:=\ldots:=X_r:=X(G,L)$.
In $O(\log_d m_0) = O(\log \log N)$ applications of (a) with $T \leq m_0$, we conclude that
$X(G,L)$ is $(\frac12, m_0, (\log N)^{-O(1)})$-spread.
In $O(\log \log N)$ additional applications of (b) with $T \geq m_0$, we conclude that
$X(G,L)$ is $(\frac12, \Omega(N), (\log N)^{-O(1)})$-spread.

Since
$X(G,L)$ is an $(\Omega(N),(\log N)^{-O(1)})$-spread subspace,
the statement of Theorem \ref{thm:pseudorandom} immediately follows from
Lemma \ref{lem:spread_vs_distortion}(\ref{1}.
\remove{
Let $\epsilon >0$ be a sufficiently small constant; then by setting first
\bigskip
\begin{equation}\label{eq:recursion}
  t_i=N\cdot \epsilon^{\frac 12\of{\frac 32}^i+\frac 12},
\end{equation}
(as in Lemma \ref{lem:iterate}) we can descend from $\epsilon N$ to $m_0$ within $O(\log\log N)$ steps; this
proves that $X(G,L)$ is an $(m_0,\epsilon N,(\log N)^{-O(1)})$-spread
subspace. And by letting $t_i=m_0\cdot (\epsilon\sqrt d)^i$ we descend from
$m_0$ to 1/2, also (due to the lower bound on $d$ in
\eqref{eq:another_bound_on_d}) within $O(\log\log N)$ steps. That is,
$X(G,L)$ is also a $(m_0,(\log N)^{-O(1)})$-spread subspace. Combining these
two facts, $X(G,L)$ is an $(\epsilon N,(\log N)^{-O(1)})$-spread subspace, and
now the statement of Theorem \ref{thm:pseudorandom} immediately follows from
Lemma \ref{lem:spread_vs_distortion}(\ref{1}.
}

\bigskip
\noindent
{\bf Proof of Theorem \ref{thm:construction}.} \label{final}
This is our most sophisticated
construction: we use a series of $X(G,L)$ for {\em different} graphs $G$, and
the ``inner'' spaces $L$ will come from Theorem
\ref{thm:guv_boosting}.
In what follows, we assume that $N$ is sufficiently large
(obviously for $N = O(1)$, every non-trivial subspace has bounded distortion).

To get started, let us denote
$$
\widetilde\eta =\frac{\eta}{(\log \log N)^2},
$$
and let us first construct and analyze subspaces $X(G,L)$ needed for our
purposes individually. For that purpose, fix (for the time being) any value
of $m$ with
\begin{equation}\label{eq:bound_on_k}
  1\leq m\leq\delta\widetilde\eta^{2\beta_0/3} N,
\end{equation}
$\delta$ a sufficiently small constant and $\beta_0$
is the constant from Theorem \ref{thm:guv_boosting}.

Applying Theorem \ref{thm:spectral}
(with $d:=N/m$), we get, for some $d=\Theta(N/m)$, an explicit
$(N,n,4,d)$-right regular graph $G_m$ with $\Lambda_{G_m}(m) \geq \Omega(d^{-1/2}) m$.
Note that \eqref{eq:bound_on_k} implies
$\widetilde \eta\geq d^{-2\beta_0/3}$ (provided the constant $\delta$ is small
enough), and thus all conditions of Theorem \ref{thm:guv_boosting} with
$N:=d,\ \eta:=\widetilde\eta$ are met. Applying that theorem, let
$L_m\subseteq \R^d$ be an explicit subspace with $\codim(L_m)\leq
\widetilde\eta d$ that is a $(d^{\frac12 + \beta_0},(\eta/\log \log N)^{O(1)})$-spread subspace. Consider
the space $X(G_m,L_m)\subseteq \R^N$.

Since $D=4$ is a constant, we have
$$
\codim(X(G_m,L_m)) \lesssim\widetilde\eta N=\frac{\eta N}{(\log \log N)^2}.
$$
And Theorem \ref{thm:pushdown} (applied to $T:=m$) implies
(recalling $\Lambda_{G_m}(m) \gtrsim d^{-1/2} m$, $t=d^{\frac12 + \beta_0}$, $d=\Theta(N/m)$) that
$X(G_m,L_m)$ is a $\of{m,\ \Omega\left(\left(\frac{N}{m}\right)^{\beta_0}\right) m,\ (\eta/\log
 \log N)^{O(1)}}$-spread subspace. We note that it is here that we crucially
use the fact that $L_m$ has spreading properties for $t \gg d^{1/2}$ ($t$
is the parameter from Definition~\ref{def:spread}) so that we more
than compensate for the factor $\sqrt{d}$ loss in
Theorem \ref{thm:pseudorandom} caused by the relatively poor expansion rate
of spectral expanders.

We will again apply Lemma \ref{lem:iterate}, but the spaces $X_i$ will
now be distinct.  In particular, for $i \in \N$ define $X_i =
X(G_{t_i},L_{t_i})$, where
$$
t_i = N \cdot \left(\frac{\eps}{N}\right)^{(1-\beta_0)^i}\ ,
$$
for some sufficiently small constant $\eps$, $0 < \eps < 1$.
It is easy to see that for some $r = O(\log \log N)$, we have
$t_r \leq \delta \tilde \eta^{2\beta_0/3} N$ and
$t_r \gtrsim \left(\delta \tilde \eta^{2\beta_0/3}\right)^2 N$.

\remove{
Now we apply essentially the same recursion as the recursion
\eqref{eq:recursion} from the proof of Theorem \ref{thm:pseudorandom}:
$$
t_i=N\cdot \epsilon^{\frac 12\of{1-\beta_0}^i+\frac 12},
$$
where (cf. \eqref{eq:bound_on_k}) $\epsilon=\delta\widetilde\eta^{2\beta_0/3}$, but
this time we go all the way down from $t_0=\epsilon N$ to $t_r=1/2$---this
still takes $O(\log\log N)$ steps. And, as we already mentioned above,
the principal difference from previous similar arguments is that the spaces
$X_i$ in Lemma \ref{lem:iterate} now become different. Namely, we let
$X_i=X(G_{t_i},L_{t_i})$.
}

Then for $X=\bigcap_{i=0}^{r-1} X_i$ we have $\codim(X)\lesssim r\frac{\eta
N}{(\log \log N)^2}\lesssim \frac{\eta N}{\log \log N}$. In particular,
$\codim(X)\leq \eta N$ for sufficiently large $N$.

By the above argument based on Theorem~\ref{thm:pushdown} and the
choice of the $t_i$'s, it is easily seen that $X_i$ is a
$(t_i,t_{i+1},(\eta/(\log \log N))^{O(1)})$-spread subspace.  By Lemma
\ref{lem:iterate}, $X$ is a $(\eps,t_r,(\eta/(\log \log
N))^{O(\log\log N)})$-spread subspace, or equivalently a
$(t_r,(\eta/(\log \log N))^{O(\log\log N)})$-spread subspace. Since we
also have $t_r \geq (\eta/(\log \log N))^{O(1)} N$, the required bound
on $\Delta(X)$ follows from Lemma
\ref{lem:spread_vs_distortion}(\ref{1}.

\section{Discussion}
\label{sec:discussion}


We have presented explicit subspaces $X \subseteq \mathbb R^N$
of dimension $(1-\eta)N$
with distortion $(\eta^{-1} \log \log N)^{O(\log \log N)}$ and,
using $N^{o(1)}$ random bits, distortion $\eta^{-O(\log \log N)}$.
It is natural to wonder whether better explicit constructions
of expanders can give rise to better bounds.  We make some
remarks about this possibility.

\begin{enumerate}
\item {\bf The GUV and CRVW expander families.}
The next two theorems essentially follow from \cite{GUV07} and
\cite{CRVW02}, respectively (after an appropriate
application of Lemma \ref{lem:right-regular}).

\begin{theorem}[\cite{GUV07}]
\label{thm:guv-regular} For each fixed $0 < c,\varepsilon \le 1$, and for all
integers $N,K$ with $K \le N$, there is an explicit construction of an
$(N,n,D,d)$-right regular graph $G$ with $D \lesssim ((\log N)/\varepsilon)^{2+2/c}$ and $d \ge
N/(D K^{1+c})$ and such that $\Lambda_G(m) \geq (1-\varepsilon) D \cdot \min\{K,m\}$.
\end{theorem}

\begin{theorem}[\cite{CRVW02}]
\label{thm:crvw}
For every fixed $0 < \varepsilon < 1$
and all sufficiently large values $N$
and $d$ there exist $n\leq N,\ D \leq 2^{O((\varepsilon^{-1} \log \log d)^3)}$ and
an explicit
$(N,n,D,d)$-right regular bipartite graph $G$
with $\Lambda_G(m) \geq (1-\varepsilon) D \cdot \min \left\{ \Omega(N/d), m \right\}$.
\end{theorem}

The main problem for us in both these constructions is that $D$ must
grow with $N$ and $d$, respectively.  By plugging in the explicit
subspaces of Theorem \ref{thm:local-subspace} into Theorem
\ref{thm:pushdown} with the GUV-expanders from Theorem
\ref{thm:guv-regular}, one can achieve distortions $\Delta(X) \approx
\exp(\sqrt{\log N \log \log N})$ for $X \subseteq \mathbb R^N$ with
$\dim(X) \geq N/2$. Using the GUV-expanders (in place of the
sum-product expanders) together with spectral expanders in a
construction similar to the proof of Theorem~\ref{thm:construction}
would yield a distortion bound of $(\log N)^{O(\log\log N)}$.

\item
{\bf Very good expansion for large sets.}
If it were possible to construct an
$(N,n,D,d)$-right regular bipartite graph $H$
with $D = O(1)$ and
such that for every $S \subseteq V_L$
with $|S| \geq N^{1-\beta}$, we had $|\Gamma(S)| = \Omega(n)$,
then we would be able to achieve $O(1)$ distortion using only $O(d^2+N^{\delta})$
random bits for any $\delta > 0$ (in fact, we could use only $O(d+N^{\delta})$ random
bits with \cite{LS07}).

The idea would be to follow the proof of Theorem \ref{thm:pseudorandom},
but only for $O(1)$ steps to show the subspace is $(N^{1-\beta}, \Omega(1))$-spread.
Then we would intersect this with a subspace $X(H,L)$, where $L \subseteq \mathbb R^d$,
with the latter subspace generated as the kernel of a random sign matrix (requiring $d^2$ bits).
Unfortunately, \cite[Th. 1.5]{RT00} shows that in order to achieve
the required expansion property, one has to take $D \geq \Omega(\beta \log N)$.
\end{enumerate}


\subsection*{Acknowledgments}

We are grateful to Avi Wigderson for several enlightening
discussions, and especially his suggestion that the sum-product
expanders of \cite{BIW,BKSSW} should be relevant.
Using the sum-product expanders in place of GUV-expanders
in Section \ref{sec:boost}, we were able to
improve our distortion bound from $(\log N)^{O(\log \log N)}$ to
$(\log N)^{O(\log \log \log N)}$. We are also thankful to an anonymous
referee for several useful remarks.

\bibliographystyle{abbrv}
\bibliography{expsec}
\end{document}